\documentclass{amsart}
\usepackage{amssymb,amsmath,amsthm}
\usepackage{enumerate}
\usepackage[matrix,arrow,curve,frame]{xy}    
\xymatrixcolsep{1.9pc}                          
\xymatrixrowsep{1.9pc}
\newdir{ >}{{}*!/-5pt/\dir{>}}                  

\def\longfib{\DOTSB\relbar\joinrel\twoheadrightarrow}

\newtheorem{thm}[subsection]{Theorem}
\newtheorem{theorem}[subsection]{Theorem}

\newtheorem{prop}[subsection]{Proposition}
\newtheorem{proposition}[subsection]{Proposition}

\newtheorem{cor}[subsection]{Corollary}
\newtheorem{corollary}[subsection]{Corollary}
\newtheorem{lemma}[subsection]{Lemma}

\theoremstyle{definition}  
\newtheorem{example}[subsection]{Example}

\newtheorem{remark}[subsection]{Remark}
\newtheorem{definition}[subsection]{Definition}

\DeclareMathOperator{\Hom}{Hom}

\DeclareMathOperator{\Ob}{Ob}

\DeclareMathOperator{\End}{End}

\DeclareMathOperator{\Mod}{Mod-}

\DeclareMathOperator{\Stmod}{Stmod}

\DeclareMathOperator{\Ho}{Ho}

\DeclareMathOperator{\id}{id}
\DeclareMathOperator{\Ev}{Ev}

\DeclareMathOperator{\Ann}{Ann}
\DeclareMathOperator{\ho}{Ho}
\DeclareMathOperator{\Vect}{Vect}
\DeclareMathOperator{\hEnd}{hEnd}
\DeclareMathOperator{\coeq}{coeq}

\newcommand{\dfn}{\textbf} 

\newcommand{\mdfn}[1]{\dfn{\mathversion{bold}#1}} 

\newcommand{\Z}{{\mathbb Z}}
\newcommand{\bZ}{{\mathbb Z}}

\newcommand{\A}{{\mathcal A}}
\newcommand{\cB}{{\mathcal B}}
\newcommand{\ucB}{\underline{\cB}}

\newcommand{\cM}{{\mathcal M}}
\newcommand{\cN}{{\mathcal N}}
\newcommand{\M}{{\mathcal M}}

\newcommand{\cC}{{\mathcal C}}

\newcommand{\cD}{{\mathcal D}}
\newcommand{\E}{\mathcal E}

\newcommand{\cV}{\mathcal V}

\newcommand{\F}{\mathbb{F}}
\newcommand{\Q}{\mathbb{Q}}

\newcommand{\EM}{{\mathcal E}_M}

\newcommand{\iso}{\cong}

\newcommand{\sm}{\wedge}

\newcommand{\Wedge}{{\scriptstyle \vee}}

\newcommand{\mc}{\colon \,}

\newcommand{\Ch}{{\mathcal C}h}
\newcommand{\dga}{dga }
\newcommand{\dgap}{dga}
\newcommand{\dgas}{dgas }
\newcommand{\dgasp}{dgas}

\newcommand{\im}{\mbox{im}}

\newcommand{\PM}{P_{\bullet}M}

\newcommand{\ra}{\rightarrow}

\renewcommand{\to}{\longrightarrow}
\newcommand{\varrow}[1]{\hbox to #1{\rightarrowfill}}
\newcommand{\varl}[2]{\stackrel{#2}{\hbox to #1{\leftarrowfill}}}

\newcommand{\varrx}[2]{\stackrel{#2}{\hbox to #1{\rightarrowfill}}}
\newcommand{\parallelarrows}[1]{\begin{array}{c} {\hbox to
#1{\rightarrowfill}}  \vspace{-0.35cm} \\ {\hbox to
#1{\rightarrowfill}} \end{array}}

\newcommand{\we}{\llra{\sim}}                   
\newcommand{\bwe}{\llla{\sim}}
\newcommand{\cofib}{\rightarrowtail}              
\newcommand{\fib}{\twoheadrightarrow}           
\newcommand{\trfib}{\stackrel{\sim}{\longfib}}
\newcommand{\trcof}{\stackrel{\sim}{\cofib}}

\newcommand{\inc}{\hookrightarrow}              
\newcommand{\dbra}{\rightrightarrows}           

\newcommand{\he}{\simeq}

\newcommand{\Set}{{\mathcal Set}}

\newcommand{\tens}{\otimes}
\newcommand{\uMI}{\underline{\cM^I}}

\newcommand{\blank}{-}
\newcommand{\adjoint}{\rightleftarrows}
\newcommand{\uM}{\underline{\cM}}
\newcommand{\cFG}{{\mathcal F}G}

\DeclareMathOperator{\Ext}{Ext}

\newcommand{\lra}{\longrightarrow}              

\newcommand{\lla}{\longleftarrow}               
\newcommand{\llra}[1]{\stackrel{#1}{\lra}}      
\newcommand{\llla}[1]{\stackrel{#1}{\lla}}      

\newcommand{\dcoprod}{\displaystyle\coprod}
\newcommand{\Me}{\cM_\epsilon}
\renewcommand{\Re}{R_\epsilon}
\newcommand{\Te}{T_\epsilon}
\newcommand{\ke}{k_\epsilon}

\newcommand{\Ae}{\A_\epsilon}
\newcommand{\uL}{\underline{L}}
\newcommand{\uR}{\underline{R}}

\numberwithin{equation}{section}

\newenvironment{myequation}
  {\addtocounter{subsection}{1}\begin{equation}}
  {\end{equation}$\!\!$}

\begin{document}

\title[Example of two model categories]
{A curious example of two model categories and
  some associated differential graded algebras}.

\date{\today; 1991 AMS Math.\ Subj.\ Class.: 18G55 (18E30, 18F25, 55U35)}

\author{Daniel Dugger}
\thanks{
}
\address{Department of Mathematics\\ University of Oregon\\ Eugene, OR 97403}
\email{ddugger@math.uoregon.edu}
\author{Brooke Shipley}
\thanks{The first author was partially supported by NSF grant
  DMS0604354.
The second author was partially supported by NSF Grant DMS0706877 and the Centre de Recerca
Matem{\`a}tica (Barcelona, Spain).}
\address{Department of Mathematics \\ 510 SEO (m/c 249)\\
University of Illinois at Chicago\\
Chicago, IL 60607
\\ USA}
\email{bshipley@math.uic.edu}

\begin{abstract}
The paper gives a new proof that the model categories of stable
modules for the rings $\Z/p^2$ and $\Z/p[\epsilon]/(\epsilon^2)$ are
not Quillen equivalent.  The proof uses homotopy endomorphism ring
spectra.  Our considerations lead to an example of two differential
graded algebras which are derived equivalent but whose associated
model categories of modules are not Quillen equivalent.  As a
bonus, we also obtain derived equivalent dgas with non-isomorphic
$K$-theories.
\end{abstract}

\maketitle

\tableofcontents

\section{Introduction}\label{sec-intro}

This paper examines two model categories $\cM$ and $\Me$, namely the
stable module categories of the rings $\Z/p^2$ and
$\Z/p[\epsilon]/(\epsilon^2)$.
 It is known from \cite{marco} that $\cM$ and $\Me$ have
equivalent homotopy categories, and that algebraic $K$-theory
computations show that $\cM$ and $\Me$ are {\it not\/} Quillen
equivalent.  Even more, by \cite{TV} it follows that the simplicial
localizations of $\cM$ and $\Me$ are not equivalent.  The point of
this paper is to explore the homotopy theory of $\cM$ and $\Me$ in
more detail, and to give a more elementary proof that they are not
Quillen equivalent.  Our proof uses homotopy endomorphism spectra
rather than algebraic $K$-theory.  Differential graded algebras come
into the picture in that the model categories $\cM$ and $\Me$ are
Quillen equivalent to modules over certain dgas.

\medskip

Throughout the paper we fix a prime $p$ and let $k=\Z/p$.  We write $R=\Z/p^2$ and
$\Re=k[\epsilon]/(\epsilon^2)$. Each of these is a Frobenius ring, in the sense that the
injectives and projectives are the same.  As explained in \cite[2.2]{hovey-book}, there is a model
category structure on the category of $R$-modules (respectively, $\Re$-modules) where the
cofibrations are the injections, the fibrations are the surjections, and the weak equivalences are
the `stable homotopy equivalences'.  For the latter, recall that two maps $f,g\colon J \ra K$ are
said to be {\it stably homotopic\/} if their difference factors through a projective; and a {\it
stable homotopy equivalence\/} is a map $h\colon J \ra K$ for which there exists an $h'\colon K\ra
J$ where the two composites are stably homotopic to the respective identities.  We write
$\Stmod(R)$ for this model category structure, and throughout the paper we write $\cM=\Stmod(R)$
and $\Me=\Stmod(\Re)$.  These are stable model categories, in the sense that the suspension
functors on the homotopy categories are self-equivalences.

It is easy to see that the homotopy categories $\Ho(\cM)$ and
$\Ho(\Me)$ are both equivalent to the category of $k$-vector spaces.
Even more, the suspension functor on both categories is isomorphic to
the identity, and so $\Ho(\cM)$ and $\Ho(\Me)$ are equivalent as
triangulated categories.  In \cite{marco} Schlichting studied the
Waldhausen $K$-theory of the finitely-generated (or compact) objects in each
category, and observed that when $p>3$ they differ starting at $K_4$.
Specifically, $K_4(\cM)\iso\Z/p^2$ whereas $K_4(\Me)\iso \Z/p\oplus
\Z/p$.  These computations follow from classical computations of the
algebraic $K$-theory of $R$ and $\Re$ from \cite{EF} and \cite{ALPS};
see also Remark~\ref{rem-schlich}.
By arguments from \cite{ds1}, this difference in $K$-theory groups
implies that $\cM$ and $\Me$ are not Quillen equivalent.  By
\cite[Cor. 1.4]{TV}, it even implies that the simplicial localizations
of $\cM$ and $\Me$ are not equivalent.

Now, $K_4$ is a fairly elaborate invariant and the computations in
\cite{EF} and \cite{ALPS} are quite involved.  Given that $\cM$ and
$\Me$ are such simple model categories, it is natural to ask for a
more down-to-earth explanation for why they are not Quillen
equivalent.  Our goal in this paper is to give such an explanation.

\medskip

Before explaining more about how we ultimately differentiate $\cM$ and
$\Me$, it seems worthwhile to point out further ways in which they are
very similar.  Every $R$-module decomposes (non-canonically) as $F\oplus
V$ where $F$ is free and $V$ is a $k$-vector space (regarded as an
$R$-module via the quotient map $R\ra k$).  Similarly, every
$\Re$-module also decomposes as the direct sum of a free module and a
$k$-vector space.  In some sense the categories of $R$-modules and
$\Re$-modules are close to being equivalent even without the model
structure, the only difference being in the endomorphisms of the free
module $R$ compared to the free module $\Re$.  But free modules are
{\it contractible\/} in $\cM$ and $\Me$!  This might lead one to
mistakenly suspect that $\cM$ and $\Me$ were Quillen equivalent.

It is well-known that the homotopy category only encodes `first-order'
information in a model category.  One place that encodes higher-order
information is the homotopy function complexes defined by Dwyer-Kan
(see \cite[Chapter 17]{H}).  It turns out that every homotopy function
complex in $\cM$ is weakly equivalent to the corresponding homotopy
function complex in $\Me$, though.  This is because $\cM$ and $\Me$
are additive  categories, and therefore their homotopy function
complexes have models which are simplicial abelian groups---in other
words, they are generalized Eilenberg-MacLane spaces.  It follows that
the only information in the homotopy type of these function complexes
is in their homotopy groups, and such information is already in the
homotopy category.

It seems clear that the difference between $\cM$ and $\Me$ has to come
from some process which considers more than just the maps between two
objects; perhaps it has something to do with composition of maps,
rather than just looking at maps by themselves.  This is the tack we
take in the present paper.

In \cite{D} it is shown that if $X$ is an object in a stable,
combinatorial model category then there is a symmetric ring spectrum
$\hEnd(X)$---well defined up to homotopy---called the \dfn{homotopy
endomorphism spectrum} of $X$.  It is proven in \cite{D} that this
ring spectrum is
invariant under Quillen equivalence.  In the present paper we first
argue that any Quillen equivalence between $\cM$ and $\Me$ must take
the object $k\in \cM$ to something weakly equivalent to the object
$k\in \Me$.  We then compute
the two homotopy endomorphism spectra of $k$ (considered as an object
of $\cM$ and as an object of $\Me$) and we prove that these are not
weakly equivalent as ring spectra.  This then proves that $\cM$ and
$\Me$ are not Quillen equivalent; see Theorem~\ref{thm-qeq}.  The
important point here is that it is the
{\it ring structures\/} on the two spectra which are not weakly
equivalent---the difference cannot be detected just by looking at the
underlying spectra.

\subsection{Connections with differential graded algebras (dgas)}

In general, computing homotopy endomorphism ring spectra is a
difficult problem.  In our case it is easier because the two model
categories $\cM$ and $\Me$ are {\it additive\/} model categories, as
defined in \cite{ds2}.  The homotopy endomorphism spectra therefore
come to us as the Eilenberg-MacLane spectra associated to certain
``homotopy endomorphism dgas'' (investigated in \cite{ds2}), and what
we really do is compute these latter objects.  Unfortunately, such
dgas are {\it not\/} invariant under Quillen equivalence, which is why
we have to work with ring spectra.  This brings us to the question of
topological equivalence of dgas---that is to say, the question of when
two dgas give rise to weakly equivalent Eilenberg-MacLane ring
spectra.  Our task is to show that the dgas arising from $\cM$ and
$\Me$ are not topologically equivalent, which we do in
Proposition~\ref{prop-nte} by using some of the techniques from \cite{ds-tpwe}.

There is another connection with dgas, which comes from homotopical tilting theory.  Each of the
model categories $\cM$ and $\Me$ is an additive, stable, combinatorial model category with a
single compact generator (the object $k$, in both cases).  Let $T$ and $\Te$ denote the homotopy
endomorphism dgas of $k$ as computed in $\cM$ and $\Me$, respectively; see Theorem~\ref{main-thm}
and Corollary~\ref{cor-qe}. By results from \cite{D, ds2, ss2, dg-s}, it follows that $\cM$ and
$\Me$ are Quillen equivalent to the model categories $\Mod T$ and $\Mod \Te$, respectively.  In
fact, in this case it is quite easy to construct the Quillen equivalences directly without
referring to the cited work above.

 We can rephrase what we know about $\cM$ and $\Me$ in terms of $T$ and
 $\Te$.  The model categories of modules $\Mod T$ and $\Mod \Te$ have
 triangulated-equivalent homotopy categories but are not Quillen
 equivalent.  It is interesting to contrast this with
 the simpler case of rings: in \cite{ds1} it is shown that if $S$ and
 $S'$ are two rings then the model categories $\Ch_S$ and $\Ch_{S'}$
 are Quillen equivalent {\it if and only if\/} they have
 triangulated-equivalent homotopy categories (that is, if and only if
 $S$ and $S'$ are {\it derived equivalent\/}).  So this result does not
 generalize from rings to dgas.

It also follows from Schlichting's $K$-theory computations
and~\cite{ds1} that the $K$-theories of $T$ and $\Te$ are
non-isomorphic for $p>3$; see Remark~\ref{rem-schlich}.  Thus $T$ and
$\Te$ are derived equivalent dgas which for $p>3$ have non-isomorphic
$K$-theories.  Again, it was proven in \cite{ds1} that this cannot
happen for ordinary rings: derived equivalent rings have isomorphic
$K$-theory groups.  So this is another result which does not
generalize from rings to dgas.

\subsection{Diagram categories}
While our use of homotopy endomorphism spectra to differentiate $\cM$
and $\Me$ is more elementary than using algebraic $K$-theory, one
could make the case that it is still not all that elementary.  The basic
question of what is different about the underlying `homotopy theory'
represented in $\cM$ and $\Me$ is perhaps still not so clear.

A different approach to these issues is the following.  For any small
category $I$, one has model structures on the diagram categories
$\cM^I$ and $\Me^I$ in which the weak equivalences and fibrations are
objectwise.  If one can find an $I$ for which $\ho(\cM^I)$ and
$\ho(\Me^I)$ are not equivalent, then this proves that $\cM$ and $\Me$
are not Quillen equivalent.  The hope is that by looking at diagram
categories one could restructure higher-order information about $\cM$ (resp.
$\Me)$ into first-order information about $\cM^I$ (resp. $\Me^I$).

It is easy to see that for all $I$ and all diagrams $D_1,D_2\in
\Me^I$, the group $\ho(\Me^I)(D_1,D_2)$ is a $\Z/p$-vector space
(the additive structure comes from the fact that $\Me^I$ is a stable
model category); see Proposition~\ref{5.2}.  It is likewise true
that for all $I$ and all
diagrams $D_1,D_2\in \cM^I$, the abelian group $\ho(\cM^I)(D_1,D_2)$
is killed by $p^2$.  By analogy with what happens in the algebraic
$K$-theory computations, one might hope to find a certain category $I$
and two diagrams $D_1$ and $D_2$ in $\cM^I$ such that
$\ho(\cM^I)(D_1,D_2)$ is not killed by $p$.  This would prove that
$\cM$ and $\Me$ are not Quillen equivalent.

So far we have not been able to find such an $I$, but we would like
to suggest this as an intriguing open problem.  Here are some simple
results to get things started, which are proved as Propositions~\ref{5.3}
and~\ref{6.11}.  (For terminology, see Sections~\ref{se:diagram} and
~\ref{se:spectral}).

\begin{prop}
\label{pr:direct}
Let $I$ be a small, direct Reedy category.  Then for any two diagrams
$D_1,D_2\in \cM^I$, the abelian group $\ho(\cM^I)(D_1,D_2)$ is a
$\Z/p$-vector space.
\end{prop}

Another thing one can prove is the following:

\begin{prop}
Let $I$ be a free category (or more generally, a category with
$\Z/p$-cohomological dimension equal to one).  Then there is a
bijection $\alpha \colon \Ob \Ho(\cM^I) \ra \Ob \Ho(\Me^I)$, with the
property that for
any two diagrams $D_1,D_2\in \Ho(\cM^I)$ the abelian groups
\[ \Ho(\cM^I)(D_1,D_2) \qquad\text{and}\qquad \Ho(\Me^I)(\alpha
D_1,\alpha D_2) \]
are $\Z/p$-vector spaces of the same dimension.
\end{prop}

The above proposition is weaker than saying that $\Ho(\cM^I)$ and
$\Ho(\Me^I)$ are equivalent as categories, but it makes it seem likely
that this is indeed the case.  The categories $0\ra 1\ra \cdots \ra n$
of $n$ composable arrows are examples of free categories.

The simplest category which has $\Z/p$-cohomological dimension greater
than one is the coequalizer category $I$ consisting of three objects
\[ \xymatrix{
0 \ar@<0.5ex>[r]\ar@<-0.5ex>[r] & 1 \ar[r] & 2
}
\]
and four non-identity maps: the three shown above, and the map which
is equal to the two composites.  This is a directed Reedy category, so
according to Proposition~\ref{pr:direct} all of the groups
$\Ho(\cM^I)(D_1,D_2)$ are $\Z/p$-vector spaces.
We have been unable to detect any differences between $\Ho(\cM^I)$ and
$\Ho(\Me^I)$ in this case.

\begin{remark}
Another approach to detecting differences between $\cM$ and $\Me$ is mentioned in~\cite{Muro}.
There Muro finds a difference in what he calls the ``cohomologically triangulated structures''
associated to $\cM$ and $\Me$, but only in the case $p=2$.  See also \cite{BM}.  It seems likely
that there is some connection between Muro's invariant and the one obtained in the present paper,
although our invariant works at all primes.
\end{remark}


\section{Background on model categories of stable modules}

In this section we establish some basic facts about the categories $\cM=\Mod
R$  and $\Me=\Mod \Re$ of $R$-modules and $\Re$-modules.  We develop
the results for $\cM$, but then remark that the proofs all work
identically for $\Me$.

\medskip

If $M$ is a module over $\Z/p^2$, let $\Gamma M$ denote $(\Ann_M
p)/pM$.  Note that this is a $\Z/p$-vector space.  Let $C_*(M)$ denote
the chain complex with $M$ in every dimension and where the
differentials are all multiplication by $p$.  So $\Gamma M$ is just
the homology of $C_*(M)$, say in dimension $0$.

\begin{lemma}
\label{le:Gamma}
Every module $M$ over $\Z/p^2$ is isomorphic (non-canonically) to a direct
sum of $\Gamma M$ and a free module.
\end{lemma}

\begin{proof}
Let $M$ be our module.  Choose a $\Z/p$-basis $\{v_i\}$ for $pM$.
For each $i$, there exists a $w_i\in M$ such that $pw_i=v_i$.
Let $F$ be the submodule generated by the $w_i$'s.  One readily checks
that the $w_i$'s are a free basis for $F$.

The inclusion $\Ann_M p \inc M$ induces a map $(\Ann_M p)/pM
\ra M/F$.  We claim this is an isomorphism.  To see this, observe that
we have a short exact sequence of chain complexes
\[ 0 \ra C_*(F) \ra C_*(M) \ra C_*(M/F) \ra 0
\]
and $C_*(F)$ is exact, because $F$ is free.  By the zig-zag lemma, one
has $\Gamma(M)\iso \Gamma(M/F)$.  But on $M/F$ multiplication by $p$
is the zero map, since $F\supseteq pM$; so $\Gamma(M/F)=M/F$.

Finally, as $M/F$ is a $\Z/p$-vector space we can choose a basis
$\{\alpha_j\}$.  Let $\pi\colon M \ra M/F$ be the quotient map. For
any $j$, there exists a $\beta_j\in M$ such that
$\pi(\beta_j)=\alpha_j$ and $p\beta_j=0$ (this is really just the zig-zag
lemma again).  This gives us a splitting for the exact sequence $0\ra
F\inc M \ra M/F\ra 0$ by sending $\alpha_j$ to $\beta_j$.
\end{proof}

\begin{remark}
Note that by the above result $M\he \Gamma(M)$ in $\Stmod(R)$, since
free modules are contractible.
\end{remark}

Let $i\colon \Vect \inc \cM$ be the map which regards every vector
space as an $R$-module via the projection $R\ra k$.  This is the
inclusion of a full subcategory.  Note that the composite $\Gamma
\circ i$ is isomorphic to the identity.

It is easy to see that if $f\colon J\ra K$ is a stable homotopy
equivalence then $\Gamma(f)$ is an isomorphism (using that $\Gamma$
takes free modules to zero).  So one has the diagram
\[ \xymatrix{
\Vect\ar[r]^i & \cM \ar[d]\ar[r]^\Gamma & \Vect \\
 & \Ho(\cM)\ar@{.>}[ur]
}
\]
where the dotted arrow is the unique extension of $\Gamma$ (which we
will also call $\Gamma$, by abuse).  Since every object in $\Ho(\cM)$
is isomorphic to a $k$-vector space, it is clear that $\Ho(\cM) \ra
\Vect$ is bijective on isomorphism classes.  It is also clear from the
diagram that $\Ho(\cM)\ra \Vect$ is surjective on hom-sets.  We will
prove below that it is actually an equivalence.

\subsection{Homotopies}
In model categories it is more common to deal with homotopies in terms
of cylinder objects rather than path objects, as the former is more
familiar.  In stable module categories it seems to be easier to deal
with path objects, however.

If $M$ is an $R$-module, let $F\ra M$ be any surjection of a free
module onto $M$.  Write $PM=M\oplus F$.  Let $i\colon M\inc PM$ be the
inclusion.  Define $\pi\colon PM \ra M\oplus M$ by having it be the
diagonal on the first summand of $PM$, and on the second summand it is
the composite $F\ra M \inc M\oplus M$, where the second map is the
inclusion into the second factor.  So the composite $M\ra PM \ra
M\oplus M$ is the diagonal, $M\ra PM$ is a trivial cofibration, and
$PM \ra M\oplus M$ is a fibration.  Therefore $PM$ is a very good path
object for $M$ in the sense of \cite{Q,hovey-book}.

It follows that for any $R$-module $J$, the natural map
\[ \coeq\Bigl ( \cM(J,PM) \dbra \cM(J,M)\Bigr  ) \ra \Ho(\cM)(J,M) \]
is an isomorphism.  The following result is immediate:

\begin{prop}
For any $k$-vector spaces $V$ and $W$, the map $\Vect(V,W) \ra
\Ho(\cM)(V,W)$ is an isomorphism.
\end{prop}

\begin{proof}
The two arrows $\cM(V,PW)\dbra \cM(V,W)$ are checked to be the same.
The main point is that the only map $V\ra W$ which factors through a
free module is the zero map.
\end{proof}

\begin{cor}
The functors $i\colon \Vect \ra \Ho(\cM)$ and $\Gamma\colon \Ho(\cM)
\ra \Vect$ are an equivalence of categories.
\end{cor}

For later use we record the following:

\begin{prop}
\label{pr:monos}
Every injection in $\cM$ is isomorphic to a direct sum of injections of the
following forms:
\[ 0 \ra k, \quad 0 \ra R,\quad \id\colon k\ra k,\quad  \id\colon R\ra
R,
\quad \text{and}\quad  p\colon k\ra R.
\]
\end{prop}

\begin{proof}
Let $j\colon M\inc N$ be an injection of $R$-modules.  We already know we can
write $M\iso F\oplus V$ for some free module $F$ and some $k$-vector
space $V$.  So up to isomorphism we can assume $M=F\oplus V$, and that
$M$ is a submodule of $N$.
Consider the map of exact sequences
\[ \xymatrix{0\ar[r] & F \ar[d]_\iso\ar[r] & M \ar@{ >->}[d]\ar[r] & M/F
\ar@{ >->}[d]\ar[r] & 0 \\
0 \ar[r] & F \ar[r] & N \ar[r] & N/F \ar[r] & 0
}
\]
The evident projection $\pi\colon M\ra F$ gives a splitting for the top exact
sequence.  Using that $F$ is injective, we can choose
a map $N\ra F$ whose restriction to $M$ is $\pi$.
 This gives a compatible splitting for the bottom
exact sequence, showing that
\[ [M\llra{j} N] \iso [F \lra F] \oplus [M/F \lra N/F].
\]
The map $\id\colon F\ra F$ is isomorphic to a
direct sum of maps $\id\colon R\ra R$.  So now replacing $M$ with
$M/F$ and $N$ with $N/F$, we can assume that the domain of $j$ is a
$k$-vector space $V$.

So now assume $j$ is a map $V\ra N$, where $V$ is a $k$-vector space.
We again know that $N$ splits as $G\oplus W$ for some free module $G$
and some $k$-vector space $W$; so up to isomorphism we can assume
$N=G\oplus W$ and that $V$ is a submodule of $N$.

Consider the map of exact sequences
\[ \xymatrix{0\ar[r] & V\cap G \ar[d]\ar[r] & V \ar@{
    >->}[d]\ar[r] & V/(V\cap G)
\ar@{ >->}[d]\ar[r] & 0 \\
0 \ar[r] & G \ar[r] & N \ar[r] & W \ar[r] & 0.
}
\]
Since $V\cap G\inc V$ is an inclusion of vector spaces, we can choose
a splitting $\pi$.  And then again using that $G$ is injective, we can
choose a compatible splitting $N\ra G$.  So this shows
\[ [V\llra{j} N]\iso [V\cap G \inc G] \oplus [V/(V\cap G) \lra W].
\]
The second map on the right is an inclusion of $k$-vector spaces, and so up to
isomorphism it is a direct sum of maps $\id\colon k\ra k$ and $0\ra
k$.  So we are reduced to analyzing the first map on the right, which
has the form $U\ra G$ where $U$ is a $k$-vector space and $G$ is free.

Up to isomorphism we have that $U$ is a direct sum of $k$'s.  Using
the inclusion $k\ra R$ sending $1\mapsto p$, we therefore obtain an
embedding $U\inc H$ where $H$ is a free module and the image of $U$ is
$pH$.  Since $G$ is injective, there is a map $H\ra G$ extending $U\inc
G$.  It is easy to see that $H\ra G$ is also an injection.

So finally, consider the map of exact sequences
\[ \xymatrix{0\ar[r] & U \ar[d]\ar[r] & U \ar@{
    >->}[d]\ar[r] & 0
\ar@{ >->}[d]\ar[r] & 0 \\
0 \ar[r] & H \ar[r] & G \ar[r] & G/H \ar[r] & 0.
}
\]
The bottom row is split (since $H$ is injective), and so there is a
splitting $G/H \ra G$ which is trivially compatible with the splitting
$0\ra U$ of the top row.  So this shows
\[ [U\ra G]\iso [U\ra H] \oplus [0\ra G/H].
\]
The first map on the right is isomorphic to a direct sum of maps $k\ra
R$ (by construction).  Since $G/H$ is a direct summand of the free
module $G$, it is itself free.  So the second map on the right is
isomorphic to a direct sum of maps $0\ra R$, and we are done.
\end{proof}

\subsection{The case of $\Re$-modules}
All the results in the previous section have analogs for $\Me$, and
the proofs are essentially the same except replacing all occurences of
``$p$'' by ``$\epsilon$''.  For instance, if $M$ is an $\Re$-module
then we define $\Gamma(M)=(\Ann_M \epsilon)/\epsilon M$.
If anything, the proofs are slightly easier in the $\Me$ case because
every module is also a $k$-vector space.

\subsection{Equivalences}

To say that two model categories $\cC$ and $\cD$ are Quillen
equivalent means that there is a zig-zag
\[ \cC=\cC_1 \we \cC_2 \bwe \cC_3 \we \cdots \bwe \cC_n=\cD \]
of Quillen equivalences between $\cC$ and $\cD$.  (Here we are
regarding a Quillen pair $L\colon\cM\adjoint \cN\colon R$ as a map of
model categories in the direction of the left adjoint.)  The derived
functors of each Quillen equivalence induce an equivalence of the
respective homotopy categories, and by composing these equivalences we
obtain an equivalence $\Ho(\cC)\he \Ho(\cD)$.

It is sometimes confusing to have $k$ denote both an $R$-module and an
$\Re$-module.   In these cases we will write $\ke$ to indicate $k$
thought of as an $\Re$-module.

\begin{prop}
\label{pr:ktok}
Suppose that one has a zig-zag of Quillen equivalences between $\cM$
 and $\Me$.  Then under the derived equivalence of homotopy
 categories, the object $k\in \Ho(\cM)$ maps to an object isomorphic
 to $\ke \in \Ho(\Me)$.
\end{prop}

\begin{proof}
Recall that $\Ho(\cM)$ and $\Ho(\Me)$ are both isomorphic to the
category $\Vect$ of $k$-vector spaces.  There is only one object (up
to isomorphism) in this category whose set of endomorphisms has
exactly $p$ elements.
\end{proof}


\section{Stable module categories and differential graded modules}

One of our goals is to show that the model categories $\cM$ and $\Me$
are each Quillen equivalent to the model category of modules over
certain dgas.  In this section we set up the basic machinery for these
Quillen equivalences, working in slightly greater generality.

\medskip

Let $T$ be a Frobenius ring; a ring such that the projective and
injective $T$-modules coincide.  Consider $\Stmod(T)$, the stable
model category on $T$-modules from~\cite[2.2.12]{hovey-book}.  Here
the cofibrations are the injections, the fibrations are the
surjections, and the weak equivalences are the stable homotopy equivalences
as described in the introduction.
For two $T$-modules $M$ and $N$, denote by $[M,N]$ the stable homotopy
classes of maps.

The goal of this section is to show that $\Stmod(T)$ is Quillen
equivalent to a model category of dg-modules over a \dga if
$\Stmod(T)$ has a {\em compact, (weak) generator} (see below).  This
extends to the model category level certain triangulated equivalences
from~\cite{keller-dg}.

\begin{definition}
An object $M$ in $\Stmod(T)$ is \dfn{compact} if $\oplus_i[M, N_i] \to
[M, \oplus_i N_i]$ is an isomorphism, for every collection
of objects $N_i$.  $M$ is a \dfn{(weak) generator}
if $[M, N]_* = 0$ implies $N$ is weakly equivalent to $0$.
\end{definition}

\begin{lemma}
\label{le:compact}
If $M$ is stably equivalent to a finitely generated module, then
$M$ is compact in $\Stmod(T)$.
\end{lemma}

\begin{proof}
It is enough to check that every finitely-generated module is compact,
and we leave this to the reader.
\end{proof}

It follows from results of \cite{D, ds2, ss2, dg-s} that if an additive, stable, combinatorial
model category has a compact weak generator then it is Quillen equivalent to the model category of
modules over a dga (perhaps through a zig-zag of Quillen equivalences). Rather than invoke the
heavy machinery from those sources, however, it is easier in the case of $\Stmod(T)$ to just
establish the Quillen equivalence directly.  We do this next.

Define the endomorphism \dga associated to any object in
$\Stmod(T)$ as follows.  First, we need to fix projective covers and injective
hulls for each $T$-module.  To be specific we use the functorial
cofibrant and fibrant replacements coming from the small object
argument and the cofibrantly-generated model category
structure~\cite[2.1.14]{hovey-book}.

\begin{definition}
Define $I(M)$ by functorially factoring $M \to 0$ as a composite $M
\cofib I(M) \trfib 0$, a cofibration followed by a trivial fibration.
Similarly, define $P(M)$ by functorially factoring $0\to M$ as
$0\trcof P(M)\fib M$, a trival cofibration followed by a fibration.

Define $\Sigma M$ to be the cokernel of $M \to I(M)$.  Define $\Omega
M$ to be the kernel of $P(M) \to M$.  Let $[M, N]_*$ be the graded
stable homotopy classes of maps in $\Ho(\Stmod(T))$, so that
$[M,N]_n\iso [\Sigma^n M, N] \iso [ M, \Omega^n N]$.
\end{definition}

To move from the setting of $T$-modules to differential graded modules
we consider complete resolutions.  A \dfn{complete resolution} of $M$
is an acyclic $\bZ$-graded chain complex $P$ of projective (also
injective) $T$-modules together with an isomorphism between $M$ and
$Z_{-1}P$, the cycles of $P$ in degree $-1$.  Considering $M$ and
$\Omega M$ as complexes concentrated in degree zero, observe that
there is a canonical map of complexes $\pi\colon P \to M$ obtained
from the projection $P_0\ra Z_{-1}P$.  One can make a map of complexes
$i\colon \Omega M \to P$ by lifting $P(M)\ra M$ to a map $P(M)\ra
P_0$, but this lifting is not canonical; however, the map $\Omega M
\ra P$ is canonical up to chain homotopy.

One way to form a complete resolution is to take $P_{n}$ to
be $I(\Sigma^{-(n+ 1)} M)$ for $n < 0$ and for $n \geq 0$ to take
$P_{n}$ to be $P(\Omega^{n} M)$ with the obvious differentials:
\[ \xymatrixcolsep{0.8pc}\xymatrix{
& P(\Omega M) \ar@{->>}[dr]\ar[rr] && P(M) \ar@{->>}[dr]\ar[rr] && I(M)
\ar@{->>}[dr]\ar[rr] && I(\Sigma M) \ar@{->>}[dr]\ar[rr] && {} \\
\Omega^2M \ar[ur]
&&\Omega M \ar@{ >->}[ur] && M\ar@{ >->}[ur] && \Sigma M\ar@{ >->}[ur]
&&  \Sigma^2 M
}
\]
Denote this particular complete resolution by $\PM$.

\begin{definition}
Let $\Ch_T$ be the category of $\bZ$-graded chain complexes of $T$-modules.
Given $X, Y$ in $\Ch_T$ define $\Hom(X,Y)$ in $\Ch_{\bZ}$ as the
complex with $\Hom(X,Y)_n = \prod_k \hom_T(X_k, Y_{n+k})$, the set of degree $n$
maps (ignoring the differentials).    For $f=(f_k)\in \Hom(X,Y)_n$
define $df\in \Hom(X,Y)_{n-1}$ to be the tuple whose component in
$\hom_T(X_k,Y_{n+k-1})$ is $d_Yf_k+(-1)^{n+1}f_{k-1}d_X$.
Notice that $\Hom(X,X)$ is a differential graded algebra.
\end{definition}

We define $\EM = \Hom(\PM, \PM)$, the {\em endomorphism \dgap} of $M$.
It follows from Lemma~\ref{lem-p-p} below that $H_*\EM \iso [M,M]_*$,
the graded ring of stable homotopy classes of self maps of $M$.  We
denote by $\Mod\EM$ the category of right differential graded modules
over the \dga $\EM$.  This has a model category structure where the
weak equivalences are the quasi-isomorphisms and the fibrations are
the surjections.

Note that if $N$ is a $T$-module then $\Hom(\PM,N)$ is a
right module over $\EM$.

\begin{thm} \label{main-thm} If $M$ is a compact, weak generator of
$\Stmod(T)$ then there is a Quillen equivalence $\Mod\EM \ra \Stmod(T)$
where the right adjoint is given by
$$\Hom(\PM, -)\mc \Stmod(T) \to \Mod\EM.$$
\end{thm}

The proof of this result will be given below.  We can better
understand the adjoint functors in the Quillen equivalence by
splitting the adjunction into two pieces:
\[ \xymatrix{
\Mod\EM \ar@<0.5ex>[r] & \Ch_T \ar@<0.5ex>[r]^-{c_0}\ar@<0.5ex>[l] &
\Stmod(T). \ar@<0.5ex>[l]^-{i_0}
}
\]
In the first adjunction, the functors are just tensor and Hom: so the
left adjoint sends a right $\EM$-module $Q$ to $Q\tens_{\EM} \PM$.
In the second adjunction, the right adjoint $i_0$ sends a module $N$ to the
chain complex with $N$ concentrated in degree $0$.  So its left
adjoint $c_0$ sends a chain complex $P$ to $P_0/\im(P_1)$.
Thus, the left adjoint in our Quillen equivalence is the functor
\[ Q \mapsto c_0(Q\tens_{\EM} \PM).
\]
Note that this functor sends $\EM$ to $M$.

We need the following statements to prove the theorem.

\begin{lemma}\label{lem-p-p}
\label{3.6}
Let $M$ and $N$ be $T$-modules and let $P$ be a complete resolution of
$M$.
\begin{enumerate}[(a)]
\item There are isomorphisms $H_k\Hom(N,P) \iso [N, \Omega
 M]_k$, natural in $N$, for all $k\in \Z$.
\item There are isomorphisms
$H_k\Hom(P,N) \iso [M, N]_k$, natural in $N$, for all $k\in \Z$.
\item
The map $\pi_*\colon \Hom(P,P) \to \Hom(P, M)$, induced by the
map of complexes
$\pi\colon P\ra M$, is a quasi-isomorphism.
\item The map $i_*\colon \Hom(P,P) \to \Hom(\Omega M,P)$, induced by the
map of complexes
$i\colon \Omega M\ra P$, is a quasi-isomorphism.
\end{enumerate}
\end{lemma}

\begin{proof}
We can lift the isomorphism $M\ra Z_{-1}P$ to a map of complexes $\PM
\ra P$.  This gives a map $f\colon \Sigma^k M \ra
P_{-k}/\im(P_{-k+1})$, which is a weak equivalence in $\Stmod(T)$.
Any chain map $P\ra N$ of degree $k$ induces a map $\Sigma^k M \ra N$ by
precomposition with $f$.  This gives us a natural map $H_k\Hom(P,N)\ra
[\Sigma^k M,N]$.

Similarly, we can lift our isomorphism $Z_{-1}P \ra M$ to a map $P\ra
\PM$, and this induces maps $Z_kP\ra \Omega^{k+1}M$ which are again
weak equivalences in $\Stmod(T)$.   So any chain  map $N\ra P$ of
degree $k$ induces a map $N\ra Z_kP \ra \Omega^{k+1}M$.  This gives
a natural map $H_k\Hom(N,P)\ra [N,\Omega^{k+1}M]$.

It is a routine exercise to check that these two natural maps are
isomorphisms.

For part (c), first recall that any map from a projective
complex $Q$ to a bounded below acyclic complex $C$ is chain homotopic
to zero (this follows from the Comparison Theorem of homological
algebra).  It follows that $\Hom(Q,C)$ is acyclic, since the cycles in
degree $k$ are chain maps $\Sigma^k Q\ra C$.  Also, any map from an
acyclic complex $C$ to a bounded above complex of injectives $I$ is
chain homotopic to zero; so $\Hom(C,I)$ is acyclic.

Now we tackle (c).  Let $F$ denote the kernel of the chain map $P\fib
M$, and consider the short exact sequence of complexes
\[ 0\ra \Hom(P,F) \ra \Hom(P,P) \ra \Hom(P,M) \ra 0.
\]
It is enough to prove that $\Hom(P,F)$ is acyclic.  But note that $F$
decomposes as the direct sum of two complexes, namely the complexes
\[ \cdots \ra P_2 \ra P_1 \ra Z_0P \ra 0 \quad\text{and}\quad
 0 \ra P_{-1} \ra P_{-2} \ra \cdots.
\]
By the observations in the previous paragraph, $\Hom(P,C)$ is
acyclic when $C$ is either of these two complexes.

Finally, let us consider (d). Here we consider the map of complexes
$Z_0 \inc P$ (where $Z_0$ is the complex concentrated entirely in
degree $0$, consisting of the zero-cycles of $P$, $Z_0P$). We'll first show
that this induces a quasi-isomorphism after applying $\Hom(\blank,P)$.

Note that there is a short exact sequence of complexes
\[ 0 \ra \Hom(P/Z_0,P) \ra \Hom(P,P) \ra  \Hom(Z_0,P) \ra 0 \]
and that
$P/Z_0$ decomposes as the
direct sum of
\[ \cdots \ra P_2 \ra P_1 \ra 0  \quad\text{and}\quad
 0 \ra P_0/Z_0 \ra P_{-1} \ra P_{-2} \ra \cdots.
\]
As in the proof of (c), one argues by the Comparison Theorem
that $\Hom(C,P)$ is acyclic when $C$ is either a bounded below complex
of projectives or a bounded above acyclic complex.  This shows that
$\Hom(P/Z_0,P)$ is acyclic, and hence $\Hom(P,P)\ra \Hom(Z_0,P)$ is a
quasi-isomorphism.

To complete the proof of (d), just note that our map $\Omega M\ra P$
factors through $Z_0$, and that the map $\Omega M \ra Z_0$ is a weak
equivalence in $\Stmod(T)$.  The result then follows from the natural
isomorphisms in (a).
\end{proof}

\begin{proof}[Proof of Theorem~\ref{main-thm}]
To show that the given functors form a Quillen pair, we check that the
right adjoint preserves fibrations and trivial fibrations.  The
fibrations in both $\Stmod(T)$ and $\Mod\EM$ are just the surjections.
Since each level in $\PM$ is projective, $\Hom(\PM, -)$ preserves
surjections.  This functor actually preserves all weak equivalences,
as this follows from Lemma~\ref{lem-p-p}(b).  In particular, it
preserves trivial fibrations.

Let $L$ and $R$ denote the left and right adjoints in our Quillen pair
$\Mod\EM \adjoint \Stmod(T)$.  Then $R(M)=\EM$, and we remarked above
Lemma~\ref{3.6}
that $L(\EM)\iso M$.  We also note that $\EM$ is a compact generator
for $\Ho(\Mod\EM)$, from which it follows by \cite[2.2.2]{ss2} that
the only localizing subcategory of $\Ho(\Mod\EM)$ containing $\EM$ is
the whole homotopy category itself.  (Recall that a localizing
subcategory is a full triangulated subcategory that is closed under
arbitrary coproducts).  A similar statement holds for
$\Ho(\Stmod(T))$, using that $M$ is a compact generator for that
category.

Let $\uL$ and $\uR$ denote the derived functors of $L$ and $R$.  Our
task is to show that these give an equivalence of homotopy categories.
We first argue that $\uR$ preserves arbitrary coproducts.  Let
$\{N_\alpha\}$ be a set of $T$-modules.  There is of course a natural
map $\oplus_\alpha (RN_\alpha) \ra R(\oplus_\alpha N_\alpha)$.  Using
that $\EM$ is a generator for $\Ho(\Mod\EM)$, it follows that this map
is an isomorphism if and only if it induces an isomorphism after
applying $[\EM,\blank]_*$.  But it is easy to check that this is the
case, using the adjunctions and the compactness of  both $\EM$ and
$L(\EM)$.

Consider the unit and counit of the derived
adjunctions
\[  \eta_X\colon X \to \uR\,\uL(X)\\
\mbox{ and } \nu_N\colon \uL\,\uR(N)\ra N.
\]
The full subcategory of $\Ho(\Mod\EM)$ consisting of all $X$ such that
$\eta_X$ is an isomorphism is a localizing subcategory---this uses the
fact that $R$ preserves coproducts.  Likewise, the full subcategory of
$\Ho(\Stmod(T))$ consisting of all $N$ such that $\nu_N$ is an
isomorphism is a localizing subcategory.  To prove that $(\uL,\uR)$
gives an equivalence of homotopy categories, it therefore suffices to
check that $\eta_{\EM}$ and $\nu_M$ are isomorphisms since
$\EM$ and $M$ are generators.

Since $\EM$ is a cofibrant $\EM$-module, $\eta_{\EM}$ is isomorphic in
$\Ho(\Mod\EM)$ to the map $\EM\ra RL(\EM)$.  But this latter map is
an isomorphism in $\Mod\EM$.

To check that $\nu_M$ is an isomorphism we need one more step.  Note
that  by Lemma~\ref{lem-p-p}(c) the map
\[ \EM=\Hom(\PM,\PM) \ra \Hom(\PM,M)=RM
\]
 is a quasi-isomorphsm.  So $\EM$
is a cofibrant-replacement for $R(M)$.  Then $\nu_M$ is isomorphic in
$\Ho(\Stmod(T))$ to the composite $L(\EM) \ra L(RM) \ra M$.  This is
readily seen to be an isomorphism of $T$-modules.
\end{proof}


\section{Proof that $\cM$ and $\Me$ are not Quillen equivalent}

In this section we apply the material from the last section to our two
stable module categories $\cM$ and $\Me$.  We compute the endomorphism
dgas of $k$ and $\ke$, and the results of the last section show that
$\cM$ and $\Me$ are Quillen equivalent to module categories over these
dgas.  Finally, we use the results of \cite{ds-tpwe} to prove that
these module categories are not Quillen equivalent.

\medskip

First we claim that both $\cM$ and $\Me$ have compact
generators.

\begin{proposition}
The module $\bZ/p$ is a compact generator for both $\Stmod(R)$
and $\Stmod(\Re)$.
\end{proposition}

\begin{proof}
First, $\bZ/p$ is compact in $\Stmod(R)$ by Lemma~\ref{le:compact}.
~\cite[2.2.1]{ss2} shows that to be a compact generator is equivalent
to asking that every localizing subcategory which contains the given
compact object is the whole category.

If a localizing subcategory of $\Ho(\Stmod(R))$ contains $\bZ/p$, then
it contains $R$ because of the exact sequence $0\ra \bZ/p \ra R \ra
\bZ/p \ra
0$.   So it contains every free module and every $\bZ/p$-vector space, and
therefore it contains every  module by Lemma~\ref{le:Gamma}.
This shows that $\bZ/p$ is a generator of $\Stmod(R)$.

The same proof shows that $\bZ/p$ is a compact generator of $\Stmod(\Re)$.
\end{proof}

Next we identify the endomorphism \dga of our chosen generator in both
cases.

\begin{proposition}\label{prop-ep2}
The \dga $\E_k$ in $\Stmod(R)$ is quasi-isomorphic to the  \dga $A$
generated over $\bZ$ by $e$ and $x$ in degree one and $y$ in degree $-1$ with
the relations $e^2=0$, $ex + xe = x^2$, $xy = yx = 1$ and the
differentials $de = p$, $dx = 0$, and $dy = 0$.  That is,
\[ A = \bZ \langle e, x, y\rangle/(e^2=0, ex + xe = x^2, xy = yx = 1,
de = p, dx = 0, dy = 0)\]
where $|e|=|x|=1$ and $|y|=-1$.
\end{proposition}

\begin{proof}
Let $P$ be the chain complex consisting of $\Z/p^2$ in every
dimension, where the differential is multiplication by $p$.  Note that
$P$ is a complete resolution for $k$.  Then the dga $\E_k$ is
quasi-isomorphic to $\Hom(P,P)$.  Write $\Hom(P,P)=\End(P)$.

For all $n\in \Z$ we have $\End(P)_n \iso \prod_{i \in \bZ}
\Hom(\Z/p^2,\Z/p^2)\iso \prod_{i\in \Z} \bZ/p^2$.  Let $f=(f_i)$
denote an element of $\End(P)_n$, where each $f_i$ is a map $P_i \ra
P_{n+i}$.  Then the $k$th entry of $df$ is the map
$p(f_k+(-1)^{n+1}f_{k-1})$.

Let $1\in \End(P)_0$ denote the tuple where $f_i=1$ for all $i$.
Let $X\in \End(P)_1$ be the tuple where $f_i=(-1)^i$, and let $Y\in
\End(P)_{-1}$ be the tuple where $f_i=(-1)^{i+1}$.   Note that
$XY=YX=1$, and $d(X)=d(Y)=0$.   Let $E\in
\End(P)_1$ be the tuple where $f_i=1$ if $i$ is even, and $f_i=0$ if
$i$ is odd.  Note that $d(E)=p\cdot 1$, $E^2=0$, and $EX+XE=X^2$.
This allows us to construct a dga map $A\ra \End(P)$ by sending $x\mapsto
X$, $y\mapsto Y$, and $e\mapsto E$.

We can uniquely write every element of $\Hom(\bZ/p^2,
\bZ/p^2)=\bZ/p^2$ in the form $a+ pb$ for $a,b \in\{ 0, \cdots,
p-1\}$.  Using this notation, the cycles in $\End(P)_n$ for $n$ even
are tuples $f$ of the form $f_i = a + pb_i$, where $a$ is independent
of $i$.  For $n$ odd the cycles are tuples satisfying $f_i= a + pb_i$
when $i$ is even, and $f_i=(p-a)+pb_i$ when $i$ is odd; here again,
$a$ is independent of $i$.  Independently of the parity of $n$, the
boundaries in each degree are tuples where every entry is a multiple
of $p$ (that is, tuples satisfying $f_i=pb_i$).  Thus we see that
$H_n(\End(P))\iso \bZ/p$ for all $n$.

Now, it is easy to verify that in degree $n$ the dga $A$ consists of
the free abelian group generated by $x^n$ and $ex^{n-1}$.  This is
valid in negative dimensions as well if one interprets
$x^{-1}$ as $y$.  This description makes it routine to check that our
map $A\ra \End(P)$ is a quasi-isomorphism.
\end{proof}

\begin{proposition}\label{prop-eke}
The \dga $\E_{\ke}$ in $\Stmod(\Re)$ is quasi-isomorphic to the formal \dga
$\Ae =  \bZ/p[x,y]/(xy-1)$
with trivial differential.  Here $|x|=1$ and $|y|=-1$.
\end{proposition}

\begin{proof}
This time let $P$ be the chain complex with $\Re$ in every dimension,
and where the differentials are all  multiplication by $\epsilon$.
This is a complete resolution of $k$, and so $\E_{\ke}$ is
quasi-isomorphic to $\End(P)$.

We again have $\End(P)_n=\prod_{i\in \Z} \Hom(\Re,\Re)\iso \prod_{i\in
\Z} \Re$, and we will denote elements by tuples $f=(f_i)$ where
$f_i\colon P_i \ra P_{n+i}$.  Then the $k$th entry of $df$ is
$\epsilon(f_k+(-1)^{n+1}f_{k-1})$.

Just as in the previous proof, we define elements $1\in
\End(P)_0$, $X\in \End(P)_1$, and $Y\in \End(P)_{-1}$.  Note that
$d(X)=d(Y)=0$, $XY=YX=1$, but this time we have $p\cdot 1=0$.  So we
get a map of dgas $\Ae \ra \End(P)$.

Every element in $\Re$ can be written uniquely in the form
$a+b\epsilon$ where $a,b\in \{0,1\ldots,p-1\}$.  Repeating the same
  analysis as in the previous proof, one finds that $H_n(\End(P))\iso
  \Z/p$ for all $n$, and that $\Ae\ra \End(P)$ is a quasi-isomorphism.
\end{proof}

\begin{corollary}\label{cor-qe}
$\Stmod(R)$ is Quillen equivalent to $\Mod A$ where $A$ is the
\dga from Proposition~\ref{prop-ep2}, and $\Stmod(\Re)$ is
Quillen equivalent to $\Mod \Ae$ where $\Ae$ is the \dga
from Proposition~\ref{prop-eke}.
\end{corollary}

\begin{proof}
This follows from Theorem~\ref{main-thm} together
with~\cite[4.3]{ss1}; the latter
shows that quasi-isomorphic \dgas have Quillen equivalent module
categories.
\end{proof}

Our goal is now the following result:

\begin{theorem}\label{thm-qeq}
$\Mod A$ and $\Mod \Ae$ are not Quillen equivalent.  Hence, $\Stmod(R)$
and $\Stmod(\Re)$ are not Quillen equivalent either.
\end{theorem}

The argument can be broken up into the following steps:
\begin{enumerate}[(1)]
\item If there were a chain of Quillen equivalences between $\Mod A$
  and $\Mod \Ae$, then the object $A$ would have to be taken to $\Ae$
  in the derived equivalence of homotopy categories.  This is by
  Proposition~\ref{pr:ktok}.
\item The categories $\Mod A$ and $\Mod \Ae$ are stable, combinatorial
  model categories.  By \cite{D}, any object $X$ in these categories has
  an associated homotopy endomorphism ring spectrum, denoted
  $\hEnd(X)$.  Then by (1) and \cite[1.4]{D}, it follows that if $\Mod
  A$ and $\Mod\Ae$ were Quillen equivalent then one would have
  $\hEnd(A) \he \hEnd(\Ae)$ as ring spectra.
\item The model categories $\Mod A$ and $\Mod \Ae$ are actually
  $\Ch(\Z)$-model categories, meaning that they are tensored,
  cotensored, and enriched over $\Ch(\Z)$.  They are therefore {\it
  additive\/} model categories, in the sense of \cite{ds2}.
  But \cite[1.5, 1.7]{ds2} then says that that the homotopy endomorphism
  spectrum for any object in such a category is weakly equivalent to
  the Eilenberg-MacLane ring spectrum associated to its endomorphism
  dga.  The endomorphism dga of $A$ is just $A$ itself, and
  likewise for $\Ae$.  So this shows that if $\Mod A$ and $\Mod \Ae$
  are Quillen equivalent, then the Eilenberg-MacLane ring spectra
  corresponding to $A$ and $\Ae$ would be weakly equivalent.  That is
  to say---in the language of \cite{ds-tpwe}---$A$ and $\Ae$ would be
  {\it topologically equivalent\/}.
\end{enumerate}

By this chain of reasoning, proving Theorem~\ref{thm-qeq} reduces to
proving that $A$ and $\Ae$ are not topologically equivalent.  To get
started, we will first prove that $A$ is not quasi-isomorphic to
$\Ae$.  This is not strictly necessary for the rest of our argument,
but it sets the stage for the more complicated argument we have to
give below.

\begin{proposition}\label{prop-not-qi}
$A$ is not quasi-isomorphic to $\Ae$.
\end{proposition}

\begin{proof}
One way to proceed would be to construct a cofibrant-replacement
$QA\trfib A$ of dgas, and then to show that there is no
quasi-isomorphism from $QA$ to $B$.  The obstruction comes from the
relation $ex + xe = x^2$.  While an argument can be  made along these
lines, we instead give a different proof which will motivate the
argument for ring spectra in  Proposition~\ref{prop-nte} below.
%

Note that if $A$ and $\Ae$ were quasi-isomorphic, then there
would be an isomorphism between the rings $H_*(\bZ/p \otimes^L_{\bZ} A)$ and
$H_*(\bZ/p \otimes^L_{\Z} \Ae)$.  Since $A$ is cofibrant as a module
over $\bZ$, we have
$H_*(\bZ/p \otimes^L_{\Z} A) \iso H_*(\bZ/p \otimes A)$, which is the ring
\[\bZ/p\langle
e, x, y; de=dx=dy=0\rangle/
(e^2=0, ex + xe = x^2, xy =yx=1)
\]
where $|e|=|x|=1$ and $|y|=-1$.
For the other case, we use $C= \bZ \langle f; df=p \rangle/(f^2)$ as a
dga which is weakly equivalent to $\bZ/p$ and also cofibrant as a
$\Z$-module.  We then
calculate that
\[ H_*(\bZ/p \otimes^L_{\Z} \Ae) \iso H_*(C \otimes \Ae) \iso
\Lambda_{k}(f) \otimes k[x,y]/(xy-1)
\]
where $|f|=|x|=1$ and $|y|=-1$.
It is easy to see that the ring $H_*(\Z/p \otimes A)$ is not
isomorphic to $H_*(C \otimes \Ae)$---for example,
the latter ring is graded-commutative but the former is not.
Thus $A$ and $\Ae$ cannot be quasi-isomorphic.
\end{proof}

Before proceeding to the next result, we need to recall a few definitions.
If $T$ is a ring spectrum, a \dfn{connective cover} for $T$ is a
connective ring spectrum $U$ together with a map $U\ra T$ which
induces isomorphisms $\pi_i(U)\ra \pi_i(T)$ for $i\geq 0$.  Standard
obstruction theory arguments show that connective covers exist, and
that any two connective covers are weakly equivalent.

If $T$ is a connective ring spectrum then we can also talk about the
\dfn{Postnikov sections} of $T$.  The $n$th Postnikov section is a
ring spectrum $U$ together with a map $T\ra U$ such that $\pi_i(U)=0$
for $i>n$ and $\pi_i(T)\ra\pi_i(U)$ is an isomorphism for $i\leq n$.
Again, a standard obstruction theory argument shows that Postnikov
sections exist and are unique up to homotopy---see \cite[2.1]{ds-post} for
a detailed discussion.

It is easy to see that if $T$ and $T'$ are weakly equivalent ring
spectra then their connective covers and Postnikov sections are also
weakly equivalent ring spectra.

If $B$ is a dga, one can  define connective covers and
Postnikov sections similarly.  It is also possible to give more
explicit chain-level models, however.  We define the connective cover
$CB$ by
\[
[CB]_i=
\begin{cases} B_i & \text{if $i > 0$}, \\
Z_0B & \text{if $i=0$, and} \\
0 & \text{if $i<0$,}
\end{cases}
\]
where $Z_0B$ denotes the zero-cylces in $B$.
Note that there is a map of dgas $CB \ra B$, and this induces isomorphisms
in homology in non-negative degrees.

Next define the $n$th Postnikov section of $CB$, denoted by
$P_n(CB)$ (or just $P_n(B)$ by abuse):
\[
[P_n B]_i=
\begin{cases} CB_i & \text{if $i < n$}, \\
CB_n/\im(CB_{n+1}) & \text{if $i=n$, and} \\
0 & \text{if $i>n$.}
\end{cases}
\]
Again note that there is a map of dgas $CB \ra P_n B$.
See~\cite[3.1]{ds-tpwe} for a more thorough discussion of Postnikov
sections for \dgasp.

If $B$ is a dga, let $HB$ denote the Eilenberg-MacLane ring spectrum
associated to $B$.  It is easy to see that $H(CB)$ is a connective
cover for $HB$, and that $H(P_nB)$ is an $n$th Postnikov section for
$H(CB)$.

\begin{proposition}\label{prop-nte}
$A$ and $\Ae$ are not topologically equivalent.
\end{proposition}

\begin{proof}
If the two \dgas $A$ and $\Ae$ were topologically equivalent then
clearly their connective covers and $n$th Postnikov sections of these
covers would also be topologically equivalent.
We will show here that $P_2 A$ and $P_2 \Ae$ are not topologically
equivalent.

The second Postnikov section of $C\Ae$ is
$P_2 \Ae \iso \Z/p[x]/(x^3)$, where $x$ has degree $1$ and $dx=0$.
For the second Postnikov section of $CA$ we can use the model
\[ P_2 A=\Z\langle e,x; de=p, dx=0\rangle /(e^2=0,ex+xe=x^2,x^3=0)
\]
where $e$ and $x$ have degree $1$ (this dga clearly has a map from
$CA$, and it has the properties of a Postnikov section).

If $P_2 A$ and $P_2 \Ae$ were topologically
equivalent, then their $H\Z/p$ homology algebras would be isomorphic;
that is, we would have an isomorphism of rings between
$\pi_* (H\Z/p \sm^L_S H(P_2 A))$ and $\pi_*(H\Z/p \sm^L_S H(P_2 \Ae))$.
We will argue that the latter ring has a nonzero element of degree $1$
which commutes (in the graded sense) with every other element of
degree $1$, whereas the former ring has no such element.

Since $\Ae$ is a $\Z/p$-algebra, $H(P_2 \Ae)$ is an $H\Z/p$-algebra.
In particular, the map $H\Z/p \ra H(P_2\Ae)$ is central.  It follows
that the map
\[ H\Z/p\sm^L_S H\Z/p \ra H\Z/p\sm^L_S H(P_2\Ae)
\]
is central,
and therefore the induced map on homotopy is also central (in the
graded sense).  If
${\mathcal A}_*$
denotes the dual Steenrod algebra $\pi_*(H\Z/p \sm^L_S H\Z/p)$,
then we are saying we have a central map
\[ \theta\colon{\mathcal A}_* \ra \pi_* (H\Z/p \sm^L_S H(P_2 \Ae)).
\]
We claim that $\theta$ is an injection in degree one.  To see this, we
only need to understand the underlying spectrum of $H(P_2\Ae)$,
and as a spectrum it is weakly equivalent to $H\Z/p\Wedge \Sigma H\Z/p\Wedge
\Sigma^2 H\Z/p$.  The fact that $\theta$ is an injection in degree one
then follows at once.

The only thing we need to know here about  ${\mathcal A}_*$
is that it is graded-commutative and has a nonzero element in degree one
($\xi_1$ for $p=2$ or $\tau_0$ for $p$ odd)~\cite{milnor}.  The image
of this element under $\theta$ gives us a nonzero central element of
the ring $\pi_*(H\Z/p \sm^L_S H(P_2 \Ae))$, lying in degree $1$.
(A little extra work shows that
$\pi_*(H\Z/p\sm^L_S H(P_2\Ae))\iso {\mathcal{A}}_*[x]/(x^3)$, but we
will not need this).

Our next step is to analyze
the graded ring $\pi_*(H\Z/p \sm^L_S H(P_2 A))$.
The unit map $S \to H\Z$ induces an algebra map
\[ \phi\colon \pi_*(H\Z/p
\sm^L_S H(P_2 A)) \to \pi_*(H\Z/p \sm^L_{H\Z} H(P_2 A)).
\]
We claim that $\phi$ is an isomorphism in degree one.  To see
this we only need to understand $H(P_2A)$ as an $H\Z$-module; and as
an $H\Z$-module it is weakly equivalent to $H\Z/p\Wedge \Sigma H\Z/p
\Wedge \Sigma^2H\Z/p$.  The fact that $\phi$ is an isomorphism in
degree one now follows from the fact that $\mathcal{A}_* \ra
\pi_*(H\Z/p\sm^L_{H\Z} H\Z/p)$ is an isomorphism in degrees zero and
one.

Using what we have just learned about $\phi$, it follows that if
$\pi_*(H\Z/p\sm^L_S H(P_2A))$ had a nonzero element of degree one
which commutes with all the other elements of degree one, then the
same would be true of $\pi_*(H\Z/p\sm^L_{H\Z} H(P_2A))$.  But this
latter ring is something which is easy to calculate, because
$H\bZ$-algebra spectra are modeled by dgas \cite{dg-s}.
It is isomorphic to $H_*(\Z/p \otimes^L_{\Z}
P_2 A)$, which---since $P_2 A$ is cofibrant as a $\Z$-module---is the
same as
\[ H_*(\Z/p\tens_\Z P_2 A)\iso \Z/p\langle e, x ;de=dx=0
\rangle/ (e^2=0, ex + xe = x^2, x^3 =
0).
\]
An easy check verifies that in this ring there is no nonzero
element in degree one which commutes with all others.

Thus, $P_2 A$ and $P_2 \Ae$ are not topologically equivalent.  We
conclude that $A$ and $\Ae$ are not topologically equivalent either.
\end{proof}

\begin{proof}[Proof of Theorem~\ref{thm-qeq}]
This follows immediately from Proposition~\ref{prop-nte} and reductions
(1)--(3) made after the statement of the theorem.
\end{proof}

\begin{remark}
We could have also approached the proof of Theorem~\ref{thm-qeq} by
quoting \cite[7.2]{ds-tpwe}.  This result shows that
the model categories $\Mod A$ and $\Mod \Ae$ are
Quillen equivalent if and only if there is a cofibrant, compact
generator $P \in \Mod A$ such that $\Hom_A(P,P)$ is topologically
equivalent to $\Ae$.  But such a $P$ would have $[P,P]\iso
H_0(\Ae)\iso \Z/p$, and there is only one object in $\Ho(\Mod A)$
whose set of endomorphisms has exactly $p$ elements---namely, $A$
itself.  So we would have $\Ae$ topologically equivalent to
$\Hom_A(A,A)=A$, and this is contradicted by
Proposition~\ref{prop-nte}.
Remarks (1)--(3) above essentially constitute the proof of
\cite[7.2]{ds-tpwe} in this case.
\end{remark}

Recall that dgas are said to be \dfn{derived equivalent} if there is a
triangulated equivalence between their homotopy categories of
dg-modules.  Thus, we have established that
$A$ and $\Ae$ are derived equivalent dgas whose model categories of
modules are not Quillen equivalent.

\begin{remark}\label{rem-schlich}
It is worth noting that $A$ and $\Ae$ are also derived equivalent dgas
which, for $p>3$, have non-isomorphic $K$-theories.  To see this,
recall that Schlichting~\cite[1.7]{marco} shows that the Waldhausen
$K$-theories of the stable module categories of finitely generated
modules over $R$ and $\Re$ are not isomorphic at $K_4$, provided
$p>3$.  This is based on the calculations of $K_3$ for $R$ and $\Re$
from~\cite{EF} and~\cite{ALPS}.  Schlichting actually claims his
conclusions for $p$ odd, but the calculations of $K_3(\Z/9)$
in~\cite{ALPS} are not correct (see \cite{G} for the correct answer).
Thus we exclude $p=3$ here.
Since Schlichting considered the $K$-theory of the cofibrant and
compact objects in $\Stmod(R)$ and $\Stmod(\Re)$, it follows
from~\cite[3.10]{ds1} and Corollary~\ref{cor-qe} that $K(A)$ and
$K(\Ae)$ are not isomorphic for $p > 3$.
\end{remark}


\section{Diagram categories}
\label{se:diagram}

Note that $\cM$ and $\Me$ are cofibrantly-generated model categories.
So for any small category $I$, there are {\em projective model
category structures} on the diagram categories $\cM^I$ and $\Me^I$
where in each case the weak equivalences and fibrations are
objectwise. See \cite[Section 11.6]{H}.  Our goal in this section is
to establish some basic comparisons between the homotopy categories
$\Ho(\cM^I)$ and $\Ho(\Me^I)$.

\medskip

We will need the following lemma.  It is well-known, but we include a
proof for the reader's convenience.

\begin{lemma}
Let $\cC$ be a pointed model category and let $Y$ be a group object in
$Ho(\cC)$.  For any object $X\in \cC$, the two evident
abelian group structures on $\Ho(\cC)(\Sigma X,Y)$ are identical.
\end{lemma}

\begin{proof}
Let $f$ and $g$ be two maps in $\Ho(\cC)(\Sigma X, Y)$. We consider
the diagram
\[ \xymatrix{
\Sigma X \ar[dr]_{\gamma}\ar[r]^-{\Delta} & (\Sigma X)\times (\Sigma X) \ar[r]^-{f\times
  g} & Y\times Y \ar[r]^-\sigma & Y.\\
& (\Sigma X)\Wedge (\Sigma X) \ar@{ >->}[u]\ar[r]_-{f\Wedge g} &
Y\Wedge Y
\ar@{ >->}[u] \ar[ur]_{\Delta}
}
\]
Here $\gamma$ is the comultiplication on $\Sigma X$ constructed by
Quillen in \cite{Q}.  The vertical maps both have the form
$(id,*)\Wedge (*,id)$.  The top and bottom composites represent the
two ways of multiplying $f$ and $g$ in $\Ho(\cC)(\Sigma X,Y)$.

The properties of a comultiplication ensure that the left triangle
commutes, and the properties of a multiplication ensure that the right
triangle commutes.  The middle square is obviously commutative, so
this finishes the proof.
\end{proof}

\begin{prop}
\label{pr:Z/p-diags-M}
\label{5.2}
Let $I$ be a small category.  Then for any two diagrams $D_1,D_2\in
\Me^I$, the abelian group $\Ho(\Me^I)(D_1,D_2)$ is a $\Z/p$-vector
space.  For any two diagrams $E_1,E_2\in \cM^I$, the abelian group
$\Ho(\cM^I)(E_1,E_2)$ is killed by $p^2$.
\end{prop}

\begin{proof}
We give the proof for $\Me$, and note that the proof for $\cM$ is
similar.

First note that every diagram $D\in \Me^I$ is an abelian group object,
using the objectwise addition $D(i)\oplus D(i) \ra D(i)$.  We can
therefore study the group structure on $\Ho(\Me^I)(D_1,D_2)$ induced
by the target.  In this group structure, if $f$ is any map in
$\Ho(\Me^I)(D_1,D_2)$ then $n[f]=f+f+\cdots+f$ ($n$ times) is the same
as $\bigr (n[\id_{D_2}]\bigl )\circ f$.

However, $p[\id_{D_2}]$ is actually equal to the zero map in $\Me^I$ (even before
going to the homotopy category).  So $p[f]$ is also zero.
\end{proof}

It is natural to wonder whether there exists a small category $I$ and
diagrams $D_1,D_2\colon I\ra \cM$ such that $\Ho(\cM^I)(D_1,D_2)$ is
not a $\Z/p$-vector space.  So far we have not been able to find such
examples.  We'll next describe a  result showing that for simple
categories $I$ such examples do not exist.

A \dfn{direct Reedy category} is a category $I$ in which every
object can be assigned a non-negative integer (called its degree) such
that every non-identity morphism raises degree \cite[Def. 15.1.2]{H}.
This is a special case of the more general notion of {\it Reedy
category\/}.

If $I$ is a Reedy category and $\cC$ is a model category, then there
is a {\it Reedy model structure} on $\cC^I$, defined in
\cite[15.3]{H}.  The weak equivalences are the objectwise weak
equivalences, and when $\cC$ is cofibrantly-generated this model
structure is Quillen equivalent to the projective model structure on
$\cC^I$.  When $I$ is a direct Reedy category then the Reedy
fibrations are precisely the objectwise fibrations, and so the Reedy
and projective model structures on $\cC^I$ coincide.  The upshot is that
this gives us a nice description of the projective
cofibrations in $\cC^I$: they are the Reedy cofibrations of
\cite[15.3.2]{H}.

\begin{prop}
\label{pr:Z/p-diags-N}
\label{5.3}
Let $I$ be a small, direct Reedy category.  Then for any two diagrams
$D_1,D_2\in \cM^I$, the abelian group $\Ho(\cM^I)(D_1,D_2)$ is a
$\Z/p$-vector space.
\end{prop}

By the same proof as for Proposition~\ref{pr:Z/p-diags-M}, the result reduces to proving that for
any diagram $D\in \M^I$ the map $p[\id_D]$ represents zero in $\Ho(\cM^I)(D,D)$.  We will prove
this using a few lemmas.

\begin{lemma}
\label{le:lifts}
Let $A\cofib B$ be a cofibration in $\cM$ and let $F\fib B$ be a
surjection where $F$ is a free module.  Then any commutative square
\[ \xymatrix{A \ar[r] \ar[d] & F\ar[d] \\
B \ar@{.>}[ur]\ar[r]^{p} & B
}
\]
(where the bottom map is multiplication-by-$p$)
has a lifting as shown.
\end{lemma}

\begin{proof}
One first verifies the lemma for the generating cofibrations, which
are  $0 \ra k$, $0\ra R$, and $k\ra R$.  The first two cases are
immediate, and the third is an easy exercise.

Now use that every monomorphism in $\cM$ is a direct sum of
monomorphisms of type $0\ra k$, $0\ra R$, $\id\colon k\ra k$,
$\id\colon R\ra R$, and $k\inc R$, by Proposition~\ref{pr:monos}.
\end{proof}

\begin{prop}
\label{pr:lem}
Let $I$ be a small, direct Reedy category.
For any diagram $D\in \cM^I$, the map $p[\id_D] \colon D \ra D$ is
null-homotopic in $\cM^I$.
\end{prop}

\begin{proof}
Notice that we may as well assume that $D$ is Reedy cofibrant in
$\cM^I$.  Choose a diagram of free modules $F$ and a surjection $F\fib
D$ (that is to say, factor the map $0\ra D$ as a trivial cofibration
followed by a fibration).  We will show that the map $p\colon D \ra D$
factors through $F$.

Choose a degree function on $I$.  For each $i\in I$ of degree $0$,
choose a factorization of $p\colon D_i \ra D_i$ through $F_i$; such
a factorization exists by the above lemma applied with $A\ra B$ being
$0\ra D_i$.

We may assume by induction that we have a partial map of
diagrams $D\ra F$ defined on the subdiagrams indexed by elements in
$I$ of degree less than $n$.  By \cite[Discussion at the end of
  Section 1.52]{H}, to extend this to the subdiagrams indexed by
elements of degree less than $n+1$ we must choose, for every object
$i\in I$ of degree $n$, a lifting in the diagram
\[\xymatrix{
L_i(D) \ar[r]\ar[d] & F_i \ar[d] \\
  D_i \ar[r]^p & D_i.
}
\]
Here $L_i(D)$ is the latching object of $D$ at $i$, and we have
implicitly used that the {\it matching\/} objects of $D$ and $F$ are all
trivial because $I$ is a direct Reedy category.

Since $D$ is Reedy cofibrant, the maps $L_i(D)\ra D_i$ are all
cofibrations.  So liftings in the above square exist by
Lemma~\ref{le:lifts}, and we are done.
\end{proof}

\begin{proof}[Proof of Proposition~\ref{pr:Z/p-diags-N}]
Immediate from Proposition~\ref{pr:lem}.
\end{proof}


\section{A spectral sequence for mapping spaces}
\label{se:spectral}

In this section we continue our comparison of $\Ho(\cM^I)$ and
$\Ho(\Me^I)$ when $I$ is a relatively simple indexing category.  We
are able to give some results in situations where the
$\Z/p$-cohomological dimension of $I$ (defined below) is less than or
equal to $1$.

\medskip

\subsection{Background}
We begin with some homological algebra.  Let $\cV$ denote the category
of vector spaces over a field $F$, and let $I$ be a small category.
Then the category of diagrams $\cV^I$ is an abelian category with
enough projectives and injectives.  So given diagrams $A,B \in \cV^I$,
one has groups $\Ext^n_{\cV^I}(A,B)$ defined in the usual way via
resolutions.

It will be convenient for us to know a little about projectives in
$\cV^I$.
For each $i\in I$, let $F_i\colon I \ra
\Set$ denote the free diagram generated at $i$; that is,
$F_i(j)=I(i,j)$ for all $j\in I$.  If $X\in \cV$, let $F_i\tens X\in
\cV^I$ denote the diagram defined by
\[ (F_i\tens X)(j) = I(i,j)\tens X = \coprod_{I(i,j)} X.\]
We will sometimes write $F_i(X)$
in place of $F_i\tens X$.

Note that for each $i\in I$ one has adjoint functors
\[ F_i(\blank)\colon \cV \adjoint \cV^I\colon \Ev_i
\]
where the right adjoint sends a diagram to its value at $i$.
It follows that for each object $X\in \cV$ and each $i\in I$, the
diagram $F_i(X)$ is projective in $\cV^I$.

Let $A\in \cV^I$.  One can show that $A$ has a canonical projective
resolution obtained by normalizing the evident simplicial object
\[
\xymatrix{
\bigoplus\limits_{i_0} F_{i_0}[A(i_0)]
&\bigoplus\limits_{i_0\ra i_1} F_{i_1}[A(i_0)]\ar[l]\ar@<-1.5ex>[l]
&\bigoplus\limits_{i_0\ra i_1\ra i_2} F_{i_2}[A(i_0)]
\ar[l]\ar@<-1ex>[l]\ar@<-2ex>[l]
& \cdots \ar@<.3ex>[l]\ar@<-.5ex>[l]\ar@<-1.3ex>[l]\ar@<-2.1ex>[l]
}
\]
This is a kind of bar resolution.
Applying $\Hom_{\cV^I}(\blank,B)$ and using the apparent adjunctions,
it follows that the groups $\Ext^n(A,B)$ can be computed as the
cohomology groups of the cochain complex associated to the
cosimplicial abelian group
\[\xymatrixcolsep{1.7pc}\xymatrix{
\prod\limits_{i_0} \cV(A(i_0),B(i_0))
\ar[r]\ar@<1ex>[r]
&
\prod\limits_{i_0\ra i_1} \cV(A(i_0),B(i_1))
\ar[r]\ar@<0.8ex>[r]\ar@<1.6ex>[r]
&
\prod\limits_{i_0\ra i_1\ra i_2} \cV(A(i_0),B(i_2))
\ar[r]\ar@<0.8ex>[r]\ar@<1.6ex>[r]\ar@<2.4ex>[r]
&
{}
}
\]
We'll call this complex $\cB_{(\cV,I)}(A,B)$.

We define the \mdfn{$F$-cohomological dimension} of $I$ to be the
smallest integer $n$ with the property that $\Ext^{n+1}(A,B)=0$ for
all $A,B\in \cV^I$.

\begin{example}
Let $G$ be a group, regarded as a category with one object.  Then an
element of $\cV^G$ is just a representation of $G$, and we are dealing
with the usual homological algebra of representations.  So for
instance the group $G=\Z/2$ has cohomological dimension equal to
$\infty$ over the field $\F_2$, because $\Ext^n(R,R)\neq 0$ for
all $n$ where $R$ denotes the trivial representation of $G$ on $\F_2$.
The cohomological dimension over $\Q$ is equal to zero.
\end{example}

\begin{example}\label{6.3}
If $G$ is a directed graph on a set $S$, one may speak of the \mdfn{free category} $\cFG$
generated by $G$.  This is the category with object set equal to $S$, and whose morphisms are
formal compositions of the edges in $G$.  In the algebra literature $G$ is called a quiver, and a
diagram in $\cV^{\cFG}$ is called a representation of this quiver.  It is known that the free
categories $\cFG$ have $F$-cohomological dimension less than or equal to $1$, for every field $F$.

For each $n$, let $[n]$ denote the usual category of $n$-composable
maps  $0\ra 1 \ra \cdots \ra n$.  This is the
free category generated by the evident directed graph, and so its
cohomological dimension is less than or equal to $1$.  An easy
computation shows that it is actually equal to $1$.
\end{example}

\begin{example}
Let $I$ be the `coequalizer' category consisting of three objects
\[ \xymatrix{
0 \ar@<0.5ex>[r]\ar@<-0.5ex>[r] & 1 \ar[r] & 2
}
\]
and four non-identity maps: the three shown above, and the map which
is equal to the two composites.
There are three basic projectives, namely $F_0(k)$, $F_1(k)$, and
$F_2(k)$.  These are the diagrams
\[ k \dbra k\oplus k \ra k, \qquad 0 \dbra k \llra{=}k,
\qquad\text{and}\qquad
0\dbra 0 \ra k.
\]
In the first diagram the two maps $k \ra k\oplus k$ are the two
canonical inclusions into the direct sum; the map $k\oplus
k \ra k$ is the coequalizer.

Any diagram of the form $[0\dbra 0 \ra V]$ is projective; it is
$F_2(V)$.  Any diagram of the form $[0\dbra V \ra 0]$ has a projective
resolution of length one: namely, the resolution $0 \ra F_2(V) \ra
F_1(V) \ra 0$.  Finally, any diagram $[V\dbra 0 \ra 0]$ has a
projective resolution of length two: the resolution has the form $0
\ra F_2(V) \ra F_1(V\oplus V) \ra F_0(V) \ra 0$.

Note that any diagram $[V_0 \dbra V_1 \ra V_2]$ may be
built via successive extensions of the three types of diagrams
considered in the last paragraph.  Namely, one has  short exact
sequences
\[ 0 \ra [0\dbra 0 \ra V_2] \ra [V_0\dbra V_1 \ra V_2] \ra [V_0 \dbra
  V_1 \ra 0] \ra 0
\]
\[ \text{and}\qquad\qquad
0 \ra [0 \dbra V_1 \ra 0] \ra [V_0 \dbra V_1 \ra 0] \ra [V_0 \dbra
  0 \ra 0]\ra 0.
\]
It follows easily that $\Ext^n(D,E)=0$ for any $n>2$ and any diagrams
$D,E \in \cV^I$.

A simple computation shows that if $D=[k \dbra 0 \ra 0]$ and
$E=[0\dbra 0\ra k]$ then $\Ext^2(D,E)=k$.  So the cohomological
dimension of $I$ is equal to $2$.
\end{example}

\subsection{The spectral sequence}

Now we return to our model categories $\cM$ and $\Me$.
If $X\in \cM$, we again let $F_i\tens X\in
\cM^I$ denote the diagram defined by
\[ (F_i\tens X)(j) = I(i,j)\tens X = \coprod_{I(i,j)} X.\]
Note that for each $i\in I$ one has a Quillen adjunction
\[ F_i\tens(\blank)\colon \cM \adjoint \cM^I\colon \Ev_i
\]
 where the right adjoint sends a diagram to its value at $i$.
Consequently, for any diagram $E\in \cM^I$ there is a natural weak
equivalence of mapping spaces
\[ \uMI(F_i\tens X,E) \he \uM(X,E(i)).
\]

Let $D\in \cM^I$.  One can form the following simplicial object:
\[
\xymatrix{
\dcoprod\limits_{i_0} F_{i_0}\tens D(i_0)
&\dcoprod\limits_{i_0\ra i_1} F_{i_1}\tens D(i_0)\ar[l]\ar@<-1.5ex>[l]
&\dcoprod\limits_{i_0\ra i_1\ra i_2} F_{i_2}\tens D(i_0)
\ar[l]\ar@<-1ex>[l]\ar@<-2ex>[l]
& \cdots \ar[l]\ar@<-.8ex>[l]\ar@<-1.6ex>[l]\ar@<-2.4ex>[l]
}
\]
One can show that the homotopy colimit of this simplicial diagram is
weakly equivalent to $D$.  It follows that for any fibrant diagram
$E\in \cM^I$, the mapping space $\uMI(D,E)$ is the homotopy limit of a
corresponding cosimplicial diagram of mapping spaces.  Using our
adjunctions mentioned above, we have that $\uMI(D,E)$ is weakly
equivalent to the homotopy limit of the cosimplicial simplicial set
\[\xymatrixcolsep{1.5pc}\xymatrix{
\prod\limits_{i_0} \uM(D(i_0),E(i_0))
\ar[r]\ar@<1ex>[r]
&
\prod\limits_{i_0\ra i_1} \uM(D(i_0),E(i_1))
\ar[r]\ar@<0.8ex>[r]\ar@<1.6ex>[r]
&
\prod\limits_{i_0\ra i_1\ra i_2} \uM(D(i_0),E(i_2))\cdots
}
\]
Call this cosimplicial simplicial set $\ucB(D,E)$.
There is a resulting spectral sequence for computing the homotopy
groups of the space $\uMI(D,E)$.

Note that each mapping space $\uM(X,Y)$ is naturally a simplicial
abelian group, so using the Dold-Kan equivalence the above
cosimplicial simplicial set can be turned into a double chain
complex.  The spectral sequence in question is just the usual spectral
sequence for a double complex.

Our next task is to identify the $E_2$-term of the spectral sequence.
This is the cohomology of the cochain complexes obtained by applying
$\pi_q$ to each object in $\ucB(D,E)$.
But note that $\pi_q\cM(X,Y)\iso
\ho(\cM)(\Sigma^q X,Y)$.   One finds that this cochain complex can be
identified with $\cB_{(\cV,I)}(\Sigma^q D,E)$ where $\cV=\Ho(\cM)$ and
we regard $\Sigma^q D$ and $E$ as diagrams
$\Sigma^q D\colon I \ra \Ho(\cM)$ and $E\colon I \ra \Ho(\cM)$.

Putting everything together, we find that our spectral sequence has
\begin{myequation}
\label{eq:spseq}
E_2^{p,q}=\Ext^p_{\cV^I}(\Sigma^q D,E) \Rightarrow
\pi_{q-p}\bigl [ \uMI(D,E) \bigr ].
\end{myequation}
With this indexing the differential $d_r$ is a map $d_r\colon
E_r^{p,q} \ra E_r^{p+r,q+r-1}$.
Note that if the $\Z/p$-cohomological dimension of $I$ is less than or
equal to $1$, then the $E_2$-term is concentrated in two adjacent
columns and the spectral sequence collapses.

\begin{remark}
Everything that we've said above applies equally well to the model
category $\Me$.  If $D$ and $E$ are diagrams in $\Me^I$, one obtains a
corresponding spectral sequence
\[ E_2^{p,q}=\Ext^p_{\cV^I}(\Sigma^q D,E) \Rightarrow
\pi_{q-p}\bigl [ \underline{\Me^I}(D,E) \bigr ].
\]
If $D$ and $E$ are diagrams in $\Vect^I$ then we can regard them as
lying both in $\cM^I$ and $\Me^I$, and so we can examine both spectral
sequences at once.  They have the same $E_2$-terms, but may have
different differentials.
\end{remark}

\subsection{An application}

The functors $\Vect \llra{j}  \cM \llra{\Gamma} \Vect$ induce
functors
\[\xymatrix{
 \Vect^I \ar[r]^j &  \cM^I\ar[d] \ar[r]^\Gamma & \Vect^I \\
& \Ho(\cM^I) \ar@{.>}[ur]_{\widetilde{\Gamma}}
}
\]
where the existence of $\widetilde{\Gamma}$ follows from the fact that
$\Gamma$ takes objectwise weak equivalences to isomorphisms.  As
$\Gamma\circ j =\id$, we have
\[ \Vect^I \inc \Ho(\cM^I) \fib \Vect^I.
\]

\begin{prop}
\label{pr:isoclasses}
If the $\Z/p$-cohomological dimension of $I$ is less than or equal to
one, then $j\colon \Vect^I \ra \Ho(\cM^I)$ is surjective on
isomorphism classes.  Said differently, every diagram $D\in \cM^I$ is
weakly equivalent to $j\Gamma(D)$.

The same thing holds with $\cM$ replaced by $\Me$.
\end{prop}

\begin{proof}
We can assume $D$ is a cofibrant diagram.  Since $j\colon \Vect \ra
\Ho(\cM)$ is an equivalence of categories, so is the induced map
$\Vect^I \ra \Ho(\cM)^I$.  So there exists a diagram $E\in \Vect^I$
such that $D$ and $E$ are isomorphic when regarded as diagrams in
$\Ho(\cM)^I$.  The rest of the proof will use obstruction theory to
produce a weak equivalence $D\ra E$.

Start by choosing a framing for the diagram $D\colon I \ra \cM$.  If
$c\cM$ denotes the category of cosimplicial objects over $\cM$, such a
framing is a functor $\widetilde{D}\colon I \ra c\cM$ taking its
values in the Reedy cofibrant objects, together with a natural
isomorphism $\widetilde{D}^0 \ra D$ (we can insist on an isomorphism
here because all objects of $\cM$ are cofibrant);
see~\cite[Section 5]{hovey-book}.  Consider the following double chain complex
of abelian groups:
\[\xymatrix{
\vdots \ar[d] & \vdots \ar[d] & \cdots \\
\prod\limits_{i_0} \cM(\widetilde{D}(i_0)^1,E(i_0)) \ar[r] \ar[d] &
\prod\limits_{i_0\ra
  i_1} \cM(\widetilde{D}(i_0)^1,E(i_1)) \ar[r]\ar[d] & \cdots \\
\prod\limits_{i_0} \cM(\widetilde{D}(i_0)^0,E(i_0)) \ar[r]  & \prod\limits_{i_0\ra
  i_1} \cM(\widetilde{D}(i_0)^0,E(i_1)) \ar[r] & \cdots \\
}
\]
The spectral sequence of (\ref{eq:spseq}) coincides with the spectral
sequence for this double complex where one first takes homology in the
vertical direction and then in the horizontal direction.

We know that $D$ and $E$ are isomorphic when regarded as diagrams
$I\ra \Ho(\cM)$.  Let $\alpha$ be such an isomorphism.  For each $i\in
I$, choose a weak equivalence $f_i\colon D(i) \ra E(i)$
representing $\alpha_i$ (we know such a weak equivalence exists
because $D(i)$ is cofibrant and $E(i)$ is fibrant).  The collection of
all these $f_i$'s represents an element $z$ in the lower left group in
the above double complex.  Our goal is to produce an element in
$H_0(\blank)$ of the total complex which has $z$ as its first
component, because this will then represent an element of
$\pi_0\uMI(\cD,E)$.

The $f_i$'s do not exactly give a map of diagrams from $D$ to $E$, but
they give a `homotopy commutative' map of diagrams.  If $z_1$ denotes
the image of $z$ under the horizontal differential in the double
complex, this precisely says that $z_1$ is the image of some element
$z_2$ under the vertical differential.  That is, for every map $c:i\ra
j$ in $I$ we can choose a homotopy between the composites $f_j\circ
D(c)$ and $E(c)\circ f_i$.

The pair $(z,z_2)$ constitutes the beginning of a $0$-cycle in the
total complex.  There are obstructions to extending it further, but
the fact that the
spectral sequence for our double complex is concentrated along the
first two columns---because of our assumption on the cohomological
dimension of $I$---shows precisely that all these obstructions vanish.
So we can construct our desired $0$-cycle, and the proof is complete.
\end{proof}

\begin{cor}
Suppose the
$\Z/p$-cohomological dimension of $I$ is less than or equal to
one.  Then every abelian group $\Ho(\cM^I)(A,B)$ is a $\Z/p$-vector
space.
\end{cor}

\begin{proof}
Let $A,B\in \cM^I$.  By Proposition~\ref{pr:isoclasses}, $B$ is weakly
equivalent to a diagram $D$ of $k$-vector spaces.  So
$\Ho(\cM^I)(A,B)\iso\Ho(\cM^I)(A,D)$.  But the identity map $\id\colon D
\ra D$ is $p$-torsion, and so by arguments used in the proof of
Proposition~\ref{pr:Z/p-diags-M} it
follows that every element of $\Ho(\cM^I)(A,D)$ is $p$-torsion as well.
\end{proof}

\begin{prop}\label{6.11}
Suppose the
$\Z/p$-cohomological dimension of $I$ is less than or equal to
one.  Then the functors $\Vect^I \ra \Ho(\cM^I)$ and $\Vect^I \ra
\Ho(\Me^I)$ are both bijections on isomorphism classes.  For every two
diagrams $A,B\in \Vect^I$, the abelian groups $\Ho(\cM^I)(A,B)$ and
$\Ho(\Me^I)(A,B)$ are isomorphic.
\end{prop}

\begin{proof}
The statement that the $j$ functors are bijections on isomorphism
classes follows from Proposition~\ref{pr:isoclasses} together with the
remarks made immediately prior to it.  For the second statement,
consider the two spectral sequences
\[ E_2^{p,q}=\Ext^p_{\cV^I}(\Sigma^q A,B) \Rightarrow
\pi_{q-p}\bigl [ \uMI(D,E) \bigr ]
\]
and
\[ E_2^{p,q}=\Ext^p_{\cV^I}(\Sigma^q A,B) \Rightarrow
\pi_{q-p}\bigl [ \underline{\Me^I}(D,E) \bigr ].
\]
Both spectral sequences are concentrated along the columns $p=0$ and
$p=1$, due to the assumption on the cohomological dimension of $I$.
So both spectral sequences collapse.  Since $\Ho(\cM^I)(A,B)$ and
$\Ho(\Me^I)(A,B)$ are both $\Z/p$-vector spaces, there are no
extension problems when passing from the $E_\infty$ terms.  The result
now follows from the fact that the $E_2$-terms of the two spectral
sequences are identical.
\end{proof}


\end{document}